\documentclass[11pt,oneside]{article}
\usepackage[centertags]{amsmath}
\usepackage{amssymb}
\usepackage{amsthm}
\usepackage{hyperref}
\newtheorem{thm}{Theorem}[section]
\newtheorem{cor}[thm]{Corollary}
\newtheorem{lem}[thm]{Lemma}
\newtheorem{prop}[thm]{Proposition}
\newtheorem{defn}[thm]{Definition}

\numberwithin{equation}{section}

\newcommand{\wbs}{\boldsymbol{w}}
\newcommand{\xbs}{\boldsymbol{x}}
\newcommand{\vbs}{\boldsymbol{v}}
\newcommand{\kbs}{\boldsymbol{k}}

\newcommand{\Ybs}{\boldsymbol{Y}}
\newcommand{\ybs}{\boldsymbol{y}}

\newcommand{\X}{\boldsymbol{X}}
\newcommand{\Real}{\mathbb R}
\newcommand{\ml}{\mathfrak{m}}
\newcommand{\mubs}{\boldsymbol{\mu}}

\newcommand{\bio}{\boldsymbol{\iota}}
\newcommand{\boeta}{\boldsymbol{\eta}}
\newcommand{\lambdabs}{\boldsymbol{\lambda}}
\newcommand{\mmodels}{\boldsymbol{\vdash}}
\newcommand{\zbs}{\boldsymbol{z}}
\newcommand{\punt}{\mathbf{.}}

\begin{document}
\title{Multivariate Bernoulli and Euler polynomials via L\'evy processes}
\author{E. Di Nardo \footnote{Dipartimento di Matematica e Informatica,
Universit\`a degli Studi della Basilicata, Viale dell'Ateneo
Lucano 10, 85100 Potenza, Italia, elvira.dinardo@unibas.it}, I.
Oliva \footnote {Dipartimento di Matematica, Universit\`a di
Bologna, Piazza di Porta S. Donato 5, 40126 Bologna, Italia,
immacolata.oliva2@unibo.it}}
\date{\today}
\maketitle
\begin{abstract}
By a symbolic method, we introduce multivariate Bernoulli and Euler polynomials 
as powers of polynomials whose coefficients involve multivariate L\'evy processes. Many properties 
of these polynomials are stated straightforwardly thanks to this representation, which
could be easily implemented in any symbolic manipulation system.  A very simple  
relation between these two families of multivariate polynomials
is provided. 
\end{abstract}

\textsf{\textbf{keywords}:}
multivariate moment, multivariate Bernoulli polynomial, multivariate Euler polynomial, 
multivariate L\'evy process, umbral calculus.

\section{Introduction} 
Quite recently many authors have obtained a moment representation for various families of
polynomials. For the multivariate Hermite polynomials $H_{\vbs}(\xbs),$ with $\vbs = (v_1, \ldots, v_d) \in \mathbb{N}_0^d$ a multi-index, i.e. a vector of nonnegative integers, the moment representation 
\begin{equation}
H_{\vbs}(\xbs)=E[(\xbs \Sigma^{-1} + i \Ybs)^{\vbs}]
\label{(momrepr)}
\end{equation}
has been given by Withers \cite{Withers} with $E$ the expectation symbol, $i$ the imaginary unit,  
$\Ybs \simeq N({\bf 0}, \Sigma^{-1})$ and $\Sigma$ a covariance matrix of full rank $d$.  Making use of the Laplace distribution and of the Gamma distribution, Sun \cite{Sun} gives a moment representation for Bernoulli polynomials, Euler polynomials and Gegenbauer polynomials in the univariate case; see also references therein.
\par
By using a symbolic method, known in the literature as the {\it classical umbral calculus} \cite{SIAM},
a different moment representation for multivariate Hermite polynomials is provided in \cite{Dinardo1},
without using the imaginary unit. Umbral methods are essentially based on a symbolic device consisting 
in dealing with sequences of numbers, indexed by nonnegative integers, where the subscripts are treated as powers. Under suitable hypothesis (see \cite{dinoli} for a detailed discussion), these sequences of numbers could be interpreted as moments of random variables (r.v.'s). 
\par 
In this paper we show how the classical umbral calculus allows us to give a moment representation like 
(\ref{(momrepr)}) also for the multivariate Bernoulli polynomials $\mathcal{B}_{\vbs}^{(t)}(\xbs)$ and the Euler polynomials $\mathcal{E}_{\vbs}^{(t)}(\xbs),$ where an additional real parameter $t \in \mathbb{R}$ is included. 
While in (\ref{(momrepr)}) the random \lq\lq part\rq\rq  is represented by the multivariate Gaussian random vector $\Ybs,$ in the representation here introduced, the random \lq\lq part\rq\rq  is represented by a multivariate L\`evy process \cite{sato}. Thanks to this representation, we point out many similarities and a new relation between these two families of polynomials. Open questions are addressed at the end of the paper.
\section{Multivariate umbral calculus}
The classical umbral calculus has reached a more advanced level compared with the
notions that we resume in this section. We only recall terminology, notation and
basic definitions strictly necessary to deal with the topic of the paper. We skip the proofs,
the reader interested in is referred to \cite{Dinardo1, Dinardoeurop}.

A univariate umbral calculus consists of an alphabet $\mathcal{A}=\{\alpha, \beta, \ldots\}$ of {\it umbrae}, and an {\it evaluation} linear functional $E: {\mathbb R}[\mathcal{A}]\mapsto {\mathbb R},$ defined on the polynomial ring ${\mathbb R}[\mathcal{A}]$
such that
\begin{description} 
\item[{\it i)}]  $E[1] = 1;$ 
\item[{\it ii)}] $E[\alpha^i \beta^j \cdots] = E[\alpha^i] E[\beta^j] \cdots$ ({\it uncorrelation property}) for distinct umbrae $\alpha, \beta, \ldots$ and nonnegative integers $i,j,\ldots.$
\end{description}
A sequence $a_0 = 1, a_1, a_2, \ldots \in \Real$ is umbrally represented by an umbra
$\alpha$ if $E[\alpha^n] = a_n,$ for all nonnegative integers $n.$ The element $a_n$ is the $n$-th \emph{moment} of
the umbra $\alpha.$ The same sequence of moments could be represented by infinitely many and distinct umbrae. 
More precisely, the umbrae $\alpha$ and $\gamma$ are said to be \emph{similar} if $E[\alpha^n]=E[\gamma^n]$ 
for all nonnegative integers $n,$ in symbols $\alpha \equiv \gamma.$ 
\par
An umbra looks like the framework of a r.v. with no reference to any probability
space. The way to recognize the umbra corresponding to a r.v. is to characterize 
the sequence of moments $\{a_n\}$. When this sequence exists, we can 
compare the moment generating function (m.g.f.) of the r.v. with the so-called 
{\it generating function} (g.f.) of the umbra, that is the formal power series 
\begin{equation}
f(\alpha,z)=1 + \sum_{n \geq 1} a_n \frac{z^n}{n!}.
\label{(gf)}
\end{equation}
If the moments of the r.v. are defined only up to some finite $m,$ then one works with sequences of $m$ elements only;
see \cite{Stanley} for further details. 
\par
Let us consider a $d$-tuple of umbral monomials $\mubs=(\mu_1, \ldots, \mu_d)$ and set
$\mubs^{\vbs} = \mu_1^{v_1} \cdots \mu_d^{v_d}.$ A sequence $\{g_{\vbs}\}_{\vbs \in \mathbb{N}_0^d} \in {\mathbb R},$ with $g_{\vbs} = g_{v_1, \ldots, v_d}$ and $g_{\bf 0} = 1,$ is
umbrally represented by the $d$-tuple $\mubs$ if $E[\mubs^{\vbs}] = g_{\vbs},$
for all $\vbs \in \mathbb{N}_0^d.$ The elements
$\{g_{\vbs}\}_{\vbs \in \mathbb{N}_0^d}$ are the {\it
multivariate moments} of $\mubs.$ If $\{\mu_i\}_{i=1}^d$ are umbral monomials 
with disjoint supports \footnote{The support of an umbral polynomial $p \in {\mathbb R}[\mathcal{A}]$ is the set of all umbrae in $\mathcal{A}$ which occur in
$p.$} then $g_{\vbs}= E[\mu_1^{v_1}] \cdots E[\mu_d^{v_d}].$ 
The g.f. of the $d$-tuple $\mubs$ is
\begin{equation}
f(\mubs, \zbs) = E[e^{\mu_1 z_1 + \cdots + \mu_d z_d}] = 1 + \sum_{k \geq 1} \sum_{\vbs \in \mathbb{N}_0^d, |\vbs|=k} g_{\vbs} \frac{\zbs^{\vbs}}{\vbs!}
\label{(genfun)}
\end{equation}
where $\zbs^{\vbs}=z_1^{v_1} \cdots z_d^{v_d}, |\vbs|=v_1 + \cdots + v_d$ and $ \vbs!=v_1! \cdots v_d!.$
Two $d$-tuples $\mubs_1$ and $\mubs_2$ are said to be {\it similar}, in symbols $\mubs_1 \equiv \mubs_2,$ if and only if $f(\mubs_1,\zbs)=f(\mubs_2,\zbs),$ that is $E[\mubs_1^{\vbs}]=E[\mubs_2^{\vbs}]$ for all $\vbs \in \mathbb{N}_0^d.$ 
They are said to be {\it uncorrelated} if and only if $E[\mubs_1^{\vbs} \mubs_2^{\wbs}]= E[\mubs_1^{\vbs}]E[\mubs_2^{\wbs}]$ for all $\vbs, \wbs \in  \mathbb{N}_0^d.$

\medskip
{\bf Multivariate Bernoulli umbra.} Let $\iota$ be the Bernoulli umbra \cite{SIAM}, that is the umbra with g.f. 
$f(\iota,z) = z/(e^z-1),$ whose $n$-th coefficient is the $n$-th Bernoulli number.  
\begin{defn} \label{a}  The \emph{multivariate Bernoulli umbra} $\bio$ is the 
$d$-tuple $(\iota, \ldots, \iota),$ with all elements equal to the Bernoulli umbra $\iota.$
\end{defn}
From Definition \ref{a} and (\ref{(genfun)}), we have 
\begin{equation}
f(\bio, \zbs) = E[e^{\iota z_1 + \cdots + \iota z_d}] = f(\iota, z_1 + \cdots + z_d) = 
\frac{z_1 + \cdots + z_d}{e^{z_1 + \cdots + z_d} - 1}.
\label{(genfunmulber)}
\end{equation}
\begin{defn}
The \emph{multivariate Bernoulli numbers} $\{B_{\vbs}^{(1)}\}_{\vbs \in \mathbb{N}_0^d}$ are the coefficients of the g.f. (\ref{(genfunmulber)}), that is $B_{\vbs}^{(1)} = E[\bio^{\vbs}].$ 
\end{defn} 
Since $E[\bio^{\vbs}]=E[\iota^{v_1} \iota^{v_2} \cdots \iota^{v_d}]$ the following result is proved. 
\begin{prop} \label{ab}
$B_{\vbs}^{(1)} =  E[\iota^{|\vbs|}]$ for all $\vbs \in \mathbb{N}_0^d.$
\end{prop}
Set $\binom{\vbs}{\kbs}=\binom{v_1}{k_1} \cdots \binom{v_d}{k_d}$  for $\vbs=(v_1, \ldots, v_d), \kbs = (k_1, \ldots, k_d) \in \mathbb{N}_0^d,$ and assume $\kbs \leq \vbs$ if and only if $k_j \leq v_j$ for all $j\in \{1,2,\ldots,d\}.$
\begin{prop}
$B_{\vbs}^{(1)} = \sum_{\kbs \leq \vbs} \binom{\vbs}{\kbs} B_{\kbs}^{(1)}$  for all 
$\vbs \in \mathbb{N}_0^d$ such that $|\vbs| > 1.$ 
\end{prop}
\begin{proof}
Let $\boldsymbol{u}$ be the $d$-tuple 
$(u,\ldots,u)$ with all elements equal to the unity umbra $u,$ that is the umbra with all moments equal to $1.$  
We have 
$$\sum_{\kbs \leq \vbs} \binom{\vbs}{\kbs} B_{\kbs}^{(1)} = \sum_{\kbs \leq \vbs} \binom{\vbs}{\kbs} E[\boldsymbol\iota^{\kbs}] = \sum_{\kbs \leq \vbs} \binom{\vbs}{\kbs} E[\boldsymbol\iota^{\kbs}] E[\boldsymbol{u}^{\vbs - \kbs}]$$
and by linearity \cite{Dinardo1}  
$$\sum_{\kbs \leq \vbs} \binom{\vbs}{\kbs} E[\boldsymbol\iota^{\kbs}] E[\boldsymbol{u}^{\vbs - \kbs}] = E\left[\sum_{\kbs \leq \vbs} \binom{\vbs}{\kbs} \boldsymbol\iota^{\kbs} \boldsymbol{u}^{\vbs - \kbs}\right] = E[(\boldsymbol\iota + \boldsymbol{u})^{\vbs}]=E[(\iota+u)^{|\vbs|}],$$
as $E[(\boldsymbol\iota + \boldsymbol{u})^{\vbs}] = E[\prod_{i=1}^d (\iota + u)^{v_i}].$ Since $E[(\iota + u)^k]= E[(\iota + 1)^k] = E[\iota^k]$ for all nonnegative $k > 1$ \cite{SIAM}, the result follows from Proposition \ref{ab} 
for $k=|\vbs|.$ 
\end{proof}
{\bf Multivariate Euler umbra.} Let $\eta$ be the Euler umbra, that is the umbra with 
g.f. $f(\eta, z) = 2 e^z/ [e^{2 z} + 1],$ whose $n$-th coefficient is the $n$-th Euler number. 
\begin{defn} \label{b}  The \emph{multivariate Euler umbra} $\boeta$ is the 
$d$-tuple $(\eta, \ldots, \eta),$ with all elements equal to the Euler umbra $\eta.$
\end{defn}
From Definition \ref{b} and (\ref{(genfun)}), we have 
\begin{equation}
f(\boeta, \zbs) = E[e^{\eta z_1 + \cdots + \eta z_d}] = f(\eta, z_1 + \cdots + z_d) = \frac{2 e^{(z_1 + \cdots + z_d)}}{e^{2(z_1 + \cdots + z_d)} + 1}.
\label{(genfunmuleul)}
\end{equation}
\begin{defn}
The \emph{multivariate Euler numbers} $\{\mathfrak{E}_{\vbs}^{(1)}\}_{\vbs \in \mathbb{N}_0^d}$ are 
the coefficients of the g.f. (\ref{(genfunmuleul)}), that is $\mathfrak{E}_{\vbs}^{(1)} = E[\boeta^{\vbs}].$ 
\end{defn} 
\begin{prop} 
$\mathfrak{E}_{\vbs}^{(1)} = E[\eta^{|\vbs|}]$ for all $\vbs \in \mathbb{N}_0^d.$
\end{prop}

\medskip
{\bf Multivariate L\'evy processes.} One feature of the classical umbral calculus is the feasibility to extend the alphabet $\mathcal{A}$ by adding new symbols \cite{SIAM}, the so-called {\it auxiliary umbrae}, whose moments depend on
moments of elements in $\mathcal{A}.$ A very important example is the so called {\it dot-product} $m \punt \alpha$ 
of a nonnegative integer $m$ and an umbra $\alpha.$ By using the exponential Bell polynomials \cite{Dinardoeurop}, 
the moments of $m  \punt \alpha$ can be expressed in terms of moments of $\alpha,$ since $m \punt \alpha$ represents
a sum of $m$ uncorrelated umbrae similar to $\alpha.$ So we have $f(m \punt \alpha, z)=[f(\alpha,z)]^m$ 
and similarly \cite{Dinardo1} $f(m \punt \mubs, \zbs)= [f(\mubs, \zbs)]^m.$ Thanks to the notion of 
auxiliary umbrae, in this last equality the integer $m$ could be replaced by a real number $t \in \Real$ 
\begin{equation}
f(t \punt \mubs, \zbs)= [f(\mubs, \zbs)]^t.
\label{(gfint)}
\end{equation}
Indeed the multivariate moments of $m \punt \mubs$ are \cite{Dinardo1}
\begin{equation}
E[(m \punt \mubs)^{\vbs}] = \sum_{\lambdabs \mmodels \vbs}
\frac{\vbs!}{\ml(\lambdabs)! \,  \lambdabs!} (m)_{l(\lambdabs)} E[\mubs_{\lambdabs}]
\label{(multmom)}
\end{equation}
where the sum is over all partitions $\lambdabs=(\lambdabs_1^{r_{1}}, \lambdabs_2^{r_{2}}, \ldots)$ 
of the multi-index $\vbs$ \footnote{A partition of
a multi-index $\vbs,$ in symbols $\lambdabs \vdash \vbs,$ is a matrix
$\lambdabs = (\lambda_{ij})$ of nonnegative integers and with no zero
columns in lexicographic order such that $\lambda_{r1}+\lambda_{r2}+\cdots+\lambda_{rk}=v_r$ for
$r=1,2,\ldots,d.$ The length $l(\lambdabs)$ of $\lambdabs$ is the number of columns of $\lambdabs.$ 
The notation $\lambdabs = (\lambdabs_{1}^{r_1}, \lambdabs_{2}^{r_2}, \ldots)$
means that in the matrix $\lambdabs$ there are $r_1$ columns equal to $\lambdabs_{1},$
$r_2$ columns equal to $\lambdabs_{2}$ and so on, with $\lambdabs_{1} <
\lambdabs_{2} < \cdots.$ We set $\ml(\lambdabs)=(r_1, r_2,\ldots).$}, $E[\mubs_{\lambdabs}] = g_{\lambdabs_1}^{r_{1}}
g_{\lambdabs_2}^{r_{2}} \cdots,$ with $g_{\lambdabs_i}$ multivariate moments of $\mubs,$ and $(m)_k$ denotes the lower factorial. From (\ref{(multmom)}), $E[(m \punt \mubs)^{\vbs}]$ results to be a polynomial $q_{\vbs}(m)$ of degree $|\vbs|$ in $m.$ Then, the symbol $t \punt \mubs$ denotes the auxiliary umbra such that 
$E[(t\punt \mubs)^{\vbs}] = q_{\vbs}(t),$ by which (\ref{(gfint)}) follows (see Proposition 2.2 in \cite{Dinardo1}). 
In particular, for the multivariate Bernoulli and Euler umbrae we have
\begin{equation}
f(t \punt \bio, \zbs) = \displaystyle{ \left( \frac{z_1 + \cdots + z_d}{e^{z_1 + \cdots + z_d} - 1} \right)^t} 
\qquad \hbox{and} \qquad f(t \punt \boeta, \zbs) = \displaystyle{ \left( \frac{2 \, e^{z_1 + \cdots + 
z_d}}{e^{2(z_1 + \cdots + z_d)} + 1} \right)^t.}
\label{genfun1}
\end{equation}
The auxiliary umbrae $t \punt \bio$ and $t \punt \boeta$ are symbolic versions
of multivariate L\'evy processes. Indeed, let  $\X = \{\X_t\}_{t \geq 0}$ be a L\'evy process on $\Real^d,$
that is a stochastic process starting from $\boldsymbol{0}$ and with
stationary and independent $d$-dimensional increments. According to the multivariate L\'evy-Khintchine 
formula \cite{sato}, if we assume that $\X$ has a convergent m.g.f. $\varphi_{\scriptscriptstyle{\X}}(\zbs)$ 
in some neighborhood of $\boldsymbol{0}$, then we have 
\begin{equation}
\varphi_{\scriptscriptstyle{\X}}(\zbs) = (\varphi_{\scriptscriptstyle{\X_1}}(\zbs))^t,
\label{(levy1)}
\end{equation}
with $\X_1 = (X_1^{(1)}, \ldots, X_d^{(1)}).$ Within the multivariate umbral calculus, if we denote by $\boldsymbol{\mu}$ the $d$-tuple such that $f(\boldsymbol{\mu},\zbs)=\varphi_{\scriptscriptstyle{\X_1}}(\zbs),$ then the auxiliary umbra $t \punt \boldsymbol{\mu}$ is the umbral counterpart of $\X.$ The auxiliary umbra $t \punt \mubs$ has various algebraic properties paralleling those of $t \punt \alpha$ \cite{Dinardo1}. We recall those
we will use later on:
\begin{equation}
t \punt (c \mubs) \equiv c (t \punt \mubs), \qquad
(t+s) \punt \mubs \equiv t \punt \mubs + s \punt \mubs, \qquad t \punt (\mubs_1 + \mubs_2) \equiv t \punt \mubs_1 +
t \punt \mubs_2  
\label{distributive}
\end{equation}
for $c,s,t \in {\mathbb R},$ with $s \ne t,$ and $\mubs_1$ and $\mubs_2$ uncorrelated $d$-tuples of umbral monomials. 
If we replace $s$ with $-t$ in the second equivalence of (\ref{distributive}), the auxiliary umbra $-t \punt \mubs$ has the remarkable property $-t \punt \mubs + t \punt \mubs \equiv \epsilon,$ where $\epsilon$ is an umbra with g.f.
$f(\epsilon,z)=1.$ The auxiliary umbra $-t \punt \mubs$ is called the {\it inverse} of $t \punt \mubs$ and 
it is such that $-t \punt \mubs \equiv t \punt (-1 \punt \mubs).$ Then also $-t \punt \mubs$ 
is a symbolic version of a multivariate L\'evy process. As example, to keep the length of the paper within bounds but also for the open questions addressed in the last section, we just show the probabilistic counterpart of $-t \punt \bio$ and $-t \punt \boeta.$ \footnote{Probabilistic counterparts of $t \punt \bio$ and $t \punt \boeta$ could be given, but the involved random variables are less
known. This goal goes beyond the aim of the paper.} Indeed the following propositions give the probabilistic interpretation of the $d$-dimensional random vector $\X_1$ in (\ref{(levy1)}) corresponding to $-1 \punt \bio$ and $-1 \punt \boeta$ respectively. 
\begin{prop} \label{int1}
The inverse $-1 \punt \bio$ of the multivariate Bernoulli umbra is the umbral 
counterpart of a $d$-tuple identically distributed to $(U,\ldots, U),$ where 
$U$ is a uniform r.v. on the interval $(0,1).$
\end{prop}
\begin{prop} \label{int2}
The inverse $-1 \punt \boeta$ of the multivariate Euler umbra is the umbral 
counterpart of a $d$-tuple identically distributed to $(X,\ldots, X),$ where 
$X=2 Y-1$ with $Y$ a Bernoulli r.v. of parameter $1/2.$ 
\end{prop}
\begin{defn} \label{11}
The $t$-th-order multivariate Bernoulli numbers $\{B_{\vbs}^{(t)}\}_{\vbs \in \mathbb{N}_0^d}$ are the multivariate moments of the multivariate umbra $t \punt \bio,$ that is $B_{\vbs}^{(t)} = E[(t \punt \bio)^{\vbs}].$
\end{defn} 
Definition \ref{11} generalizes the definition of the multivariate Bernoulli numbers  given in \cite{liu}. 
In particular we have $B_{\boldsymbol{0}}^{(0)} = E[(0 \punt \bio)^{ \boldsymbol{0}}]=1$ and 
$B_{\vbs}^{(0)} = E[(0 \punt \bio)^{\vbs}] = 0$ if  $|\vbs| > 0.$
\begin{prop}  \label{liu2B}
$B_{\vbs}^{(t)} = \sum_{\kbs \leq \vbs} \binom{\vbs}{\kbs} B_{\kbs}^{(s)} B_{\vbs - \kbs}^{(t- s)},$ for all
$s,t \in \mathbb{R}$ and $\vbs \in \mathbb{N}_0^d.$ 
\end{prop}
\begin{proof}
For $s=t,$ the proof is  straightforward. For $s,t \in \mathbb{R}$ with $s \ne t,$ 
from the second of (\ref{distributive}), we have $t \punt \bio \equiv 
(t-s)\punt \bio + s \punt \bio,$ so that $E[(t \punt \bio)^{\vbs }]=E[ \{ 
(t-s)\punt \bio + s \punt \bio\}^{\vbs}]$ for all $\vbs \in \mathbb{N}_0^d.$  
The result follows from Definition \ref{11} since $E[\{(t-s)\punt \bio + s \punt \bio\}^{\vbs}] =  
\sum_{\kbs \leq \vbs} \binom{\vbs}{\kbs} E[\{(t-s) \punt \boldsymbol\iota \}^{\vbs - \kbs}] 
E[(s \punt \boldsymbol\iota)^{\kbs}].$ 
\end{proof}
In Proposition \ref{liu2B}, set $t=0.$ We have the following corollary. 
\begin{cor} \label{liu2Bcor}
For all $s \in \mathbb{R}$ we have $\sum_{\kbs \leq \vbs} \binom{\vbs}{\kbs} B_{\kbs}^{(s)} 
B_{\vbs - \kbs}^{(-s)} = 1$ if $\vbs = \boldsymbol{0}$ otherwise being $0$. 
\end{cor} 
\begin{defn} \label{12}
The $t$-th-order multivariate Euler numbers $\{\mathfrak{E}_{\vbs}^{(t)}\}_{\vbs \in \mathbb{N}_0^d}$ are  the
multivariate moments of the multivariate umbra $t \punt \boeta,$ that is $\mathfrak{E}_{\vbs}^{(t)} = E[(t \punt \boeta)^{\vbs}].$
\end{defn} 
Definition \ref{12} generalizes the definition of the multivariate Euler numbers given in \cite{liu}. As before we have $\mathfrak{E}_{\boldsymbol{0}}^{(0)} = E[(0 \punt \boeta)^{ \boldsymbol{0}}]=1,$ and 
$\mathfrak{E}_{\vbs}^{(0)} = E[(0 \punt \boeta)^{\vbs}] = 0$ if  $|\vbs| > 0.$
Proposition \ref{liu2B} and Corollary \ref{liu2Bcor} can be restated also for $\mathfrak{E}_{\vbs}^{(t)}$
since their proofs depend only on the moment umbral representation and not on the properties of
the involved umbrae. 
%
\section{Multivariate Bernoulli and Euler polynomials}
In the classical umbral calculus, we can replace the field ${\Real}$ with 
${\Real}[x_1, \ldots, x_d],$ where $x_1, \ldots, x_d$ are indeterminates \cite{Dinardo1}. Then 
the linear operator $E$ is defined on the polynomial ring ${\Real}[x_1, \ldots, x_d][\mathcal{A}]$ with
values in ${\Real}[x_1, \ldots, x_d].$ The only hypothesis to be added on the linear operator $E$ is that if $\xbs=(x_1, \ldots, x_d)$ then $E[\xbs^{\vbs} \mubs^{\wbs}]=\xbs^{\vbs} E[\mubs^{\wbs}],$ for all $\vbs, \wbs \in \mathbb{N}_0^d.$ 
\begin{defn} \label{momrap1} {\bf (Moment representation of multivariate Bernoulli polynomials)} The multivariate Bernoulli polynomial of order $\vbs \in \mathbb{N}_0^d$ is $B_{\vbs}^{(t)}(\xbs)=E[(\xbs + t \punt \bio)^{\vbs}],$ where $\bio$ is the multivariate Bernoulli umbra and $t \in \Real.$  
\end{defn}
\begin{defn} \label{momrap2} {\bf (Moment representation of multivariate Euler polynomials)} 
The multivariate Euler polynomial of order $\vbs \in \mathbb{N}_0^d$ is $\mathcal{E}_{\vbs}^{(t)}(\xbs)
= E\left\{(\xbs + \frac{1}{2} [ t \punt (\boeta - \boldsymbol{u})])^{\vbs}\right\}$ with 
$t \in \Real, \boldsymbol{u}=(u, \ldots, u)$ a $d$-tuple with all elements equal to the 
unity umbra $u$ and $\boeta$ the multivariate Euler umbra.  
\end{defn}
Definition \ref{momrap1} and \ref{momrap2} means that the multivariate Bernoulli polynomials and the multivariate Euler polynomials are such that 
\begin{equation}
B_{\vbs}^{(t)}(\xbs)= \sum_{\kbs \leq \vbs} \binom{\vbs}{\kbs} \xbs^{\vbs-\kbs} E[(t \punt \bio)^{\kbs}], \qquad
\mathcal{E}_{\vbs}^{(t)}(\xbs) = \sum_{\kbs \leq \vbs} \binom{\vbs}{\kbs} \frac{\xbs^{\vbs - \kbs}}{2^{|\kbs|}} E[ \{ t \punt (\boeta - \boldsymbol{u})\}^{\kbs}] 
\label{(mb)}
\end{equation}
where $t \punt \bio$ and $t \punt (\boeta - \boldsymbol{u})$ are multivariate L\'evy processes. This symbolic moment representation of the coefficients 
simplifies the calculus and is computationally efficient \cite{Dinardo1, dinoli}. In particular, the definition of multivariate Euler polynomials is 
given according to the terminology first introduced by N\"{o}rlund \cite{Nor}. Indeed, from the first equivalence in (\ref{distributive})
we have $\frac{1}{2} [ t \punt (\boeta - \boldsymbol{u})] \equiv  t \punt \left[\frac{1}{2} (\boeta - \boldsymbol{u})\right].$ Thanks to Proposition \ref{int2}, we have $- 1 \punt \left[\frac{1}{2} (\boeta - \boldsymbol{u})\right] \equiv \frac{1}{2} \left[ - 1 \punt \boeta + \boldsymbol{u} \right]$ and this last symbol represents a $d$-tuple identically  distributed to $(Y,\ldots, Y),$ where $Y$ is a Bernoulli r.v. with parameter $1/2.$ 
\begin{prop} \label{1} $\mathcal{B}_{\vbs}^{(t)} (\xbs) = \sum_{\kbs \leq \vbs} \binom{\vbs}{\kbs} \xbs^{\vbs - \kbs} B_{\kbs}^{(t)}$ and  $2^{|\vbs|} \mathcal{E}_{\vbs}^{(t)} \left(\frac{1}{2} \xbs +  \frac{t}{2} \boldsymbol{1} \right) = \sum_{\kbs \leq \vbs} \binom{\vbs}{\kbs}  \xbs^{\vbs - \kbs}  \mathfrak{E}_{\kbs}^{(t)}$  with $\boldsymbol{1}$ the $d$-tuple with all elements equal to $1.$  
\end{prop}
\begin{proof}
The former equality follows from the first equality in (\ref{(mb)}), by observing that $E[(t \punt \bio)^{\vbs}] = B_{\vbs}^{(t)}.$ For the latter equality, observe that
$$ E\left[ \left(\frac{1}{2} \xbs +  \frac{t}{2} \boldsymbol{1} + \frac{1}{2} \left[ t \punt (\boeta - \boldsymbol{u})\right]\right)^{\vbs} \right] = \sum_{\kbs \leq \vbs} \binom{\vbs}{\kbs} E\left[ \left\{ \frac{t}{2} \boldsymbol{1}  - \frac{1}{2} (t \punt \boldsymbol{u}) \right\}^{\kbs} \right] \frac{E\left[ \left(\xbs + t \punt \boeta \right)^{\vbs-\kbs} \right]}{2^{|\vbs-\kbs|}}.$$
Since $E\left[ \left\{ \frac{t}{2} \boldsymbol{1}  - \frac{1}{2} (t \punt \boldsymbol{u}) \right\}^{\kbs} \right]=0$ for all $\kbs,$ except when  $\kbs = \boldsymbol{0}$ which gives $1,$ we have 
$$2^{|\vbs|} \mathcal{E}_{\vbs}^{(t)} \left(\frac{1}{2} \xbs +  \frac{t}{2} \boldsymbol{1} \right) = E\left[ \left(\xbs + t \punt \boeta \right)^{\vbs}\right] =  \sum_{\kbs \leq \vbs} \binom{\vbs}{\kbs} \xbs^{\vbs - \kbs} E[(t \punt \boeta)^{\kbs}]$$
by which the result follows. 
\end{proof}
\begin{cor} $B_{\vbs}^{(t)} = \mathcal{B}_{\vbs}^{(t)}(\boldsymbol{0})$ and $\mathfrak{E}_{\vbs}^{(t)} = 2^{|\vbs|} \, \mathcal{E}_{\vbs}^{(t)}\left(\frac{t}{2} \boldsymbol{1}\right),$ with $\boldsymbol{0}$ and $\boldsymbol{1}$ the $d$-tuples with all elements equal to $0$ and $1$ respectively.
\end{cor}
\begin{cor} \label{tsh}
$E[\mathcal{B}_{\vbs}^{(t)}(- t \punt \bio)] = E[\mathcal{B}_{\vbs}^{(t)}(t \punt (-1 \punt \bio))] = 0$ and $E[\mathcal{E}_{\vbs}^{(t)}\left(\frac{1}{2} [t \punt (\boldsymbol{u} - 1 \punt \boeta)] \right)]=0.$
\end{cor} 
Taking into account Definitions \ref{momrap1}, \ref{momrap2} and (\ref{genfun1}), the g.f. of the multivariate Bernoulli and Euler polynomials are respectively 
$$f(\xbs + t \punt \bio, \zbs) = e^{x_1 z_1 + \cdots + x_d z_d} \left(\frac{z_1 + \cdots + z_d}{e^{z_1 + \cdots + z_d} - 1}\right)^t, f\left(\xbs + \frac{1}{2} [ t \punt (\boeta - \boldsymbol{u})], \zbs \right) = 
\frac{2^t e^{x_1 z_1 + \cdots + x_d z_d}}{(e^{z_1 + \cdots + z_d} + 1)^t}.$$
\begin{prop} \label{liu7} If $\xbs=(x_1, \ldots, x_d)$ and $\ybs=(y_1, \ldots, y_d)$ are two $d$-tuples of indeterminates, then
$$\mathcal{B}_{\vbs}^{(t + s)} (\xbs + \boldsymbol{y}) = \sum_{\kbs \leq \vbs} 
\binom{\vbs}{\kbs} \mathcal{B}_{\kbs}^{(t)}(\xbs) \mathcal{B}_{\vbs - \kbs}^{(s)}(\boldsymbol{y}) \quad  \mathcal{E}_{\vbs}^{(t + s)} (\xbs + \boldsymbol{y}) = \sum_{\kbs \leq \vbs} 
\binom{\vbs}{\kbs} \mathcal{E}_{\kbs}^{(t)}(\xbs) \mathcal{E}_{\vbs - \kbs}^{(s)}(\boldsymbol{y}).$$
\end{prop}
\begin{proof} 
We replace ${\Real}[x_1, \ldots, x_d]$ with ${\Real}[x_1, \ldots, x_d, y_1, \ldots, y_d]$. From the second equivalence in (\ref{distributive}), we have
$$\mathcal{B}_{\vbs}^{(t + s)} (\xbs + \boldsymbol{y}) = E\{[(\xbs + t \punt \bio)+ (\boldsymbol{y} + s \punt \bio)]^{\vbs}\} = \sum_{\kbs \leq \vbs} \binom{\vbs}{\kbs} E[(\xbs + t \punt \bio)^{\kbs}] E[(\boldsymbol{y} + s \punt \bio)^{\vbs - \kbs}],$$ by which the former equality follows. The latter equality follows by the same arguments.
\end{proof}
\begin{cor}   
$\sum_{\kbs \leq \vbs} \binom{\vbs}{\kbs} \mathcal{E}_{\kbs}^{(t)}(\xbs) 
\mathcal{E}_{\vbs - \kbs}^{(-t)}(\xbs) = \sum_{\kbs \leq \vbs} \binom{\vbs}{\kbs} \mathcal{B}_{\kbs}^{(t)}(\xbs) 
\mathcal{B}_{\vbs - \kbs}^{(-t)}(\xbs) = 2^{|\vbs|} \xbs^{\vbs}.$ 
\end{cor}
\begin{proof} 
In Proposition \ref{liu7}, set $s=-t.$ The result follows by observing that $\mathcal{E}_{\vbs}^{(0)}(2 \xbs)
=  \mathcal{B}_{\vbs}^{(0)}(2 \xbs) = E[(2 \xbs)^{\vbs}] = 2^{|\vbs|} \xbs^{\vbs}.$
\end{proof}
\begin{prop}  $\mathcal{B}_{\vbs}^{(t)} (t \boldsymbol{1} - \xbs) = (-1)^{|\vbs|} \mathcal{B}_{\vbs}^{(t)} (\xbs),$ and $\mathcal{E}_{\vbs}^{(t)}(t \boldsymbol{1} - \xbs) = (-1)^{|\vbs|} \mathcal{E}_{\vbs}^{(t)} (\xbs).$
\end{prop}
\begin{proof}  
The former equality follows by observing that $\mathcal{B}_{\vbs}^{(t)} (t \boldsymbol{1} - \xbs)=E[(t \punt \boldsymbol{u} - \xbs + 
t \punt \bio)^{\vbs}]$ and 
$$E[(t \punt \boldsymbol{u} - \xbs + t \punt \bio)^{\vbs}] = (-1)^{|\vbs|} E[(t \punt (-\boldsymbol{u}) + \xbs + 
t \punt(- \bio))^{\vbs}] = (-1)^{|\vbs|} E\left\{[\xbs + t \punt (- (\bio + \boldsymbol{u}))]^{\vbs} \right\}.$$
Since $E[(\iota + u)^k] = (-1)^k E[\iota^k]$ for all nonnegative integers $k$ \cite{SIAM}, then 
$E[(- (\bio + \boldsymbol{u}))^{\vbs} ] = (-1)^{|\vbs|} E[(\iota + u)^{|\vbs|}] = 
E[\iota^{|\vbs|}] = E[\bio^{\vbs}].$ Then we have $- (\bio + \boldsymbol{u}) \equiv \bio$ and $t \punt (- (\bio + \boldsymbol{u})) \equiv 
t \punt \bio,$ by which the result follows. Similarly we have 
$$\mathcal{E}_{\vbs}^{(t)} \left(t \boldsymbol{1} - \xbs \right) = E\left[\left(t \boldsymbol{1} - \xbs + t \punt \frac{\boeta}{2} - \frac{t}{2} \boldsymbol{1} \right)^{\vbs} \right] 
= (-1)^{|\vbs|} E\left[\left(\xbs + t \punt \left(-\frac{\boeta}{2} \right) + t \punt \left(-\frac{\boldsymbol{u}}{2} \right) \right)^{\vbs}\right].$$ 
Since $f(-\frac{\boeta}{2},\zbs)=f(\frac{\boeta}{2},\zbs),$ the latter result follows. 
\end{proof} 
As it happens for the univariate case, the multivariate Bernoulli and Euler polynomials share many properties.
Undoubtedly, this is due to the connection between Bernoulli and Euler numbers that here is emphasized by 
the similar multivariate moment representation. Therefore it is reasonable to ask for relations between them. 
We have chosen to show a connection between the multivariate umbrae they are related to, which can be 
translated in a connection between the $t$-th-order multivariate Bernoulli and Euler numbers.  
\begin{lem} \label{c}  If $\bio$ is the multivariate Bernoulli umbra and $\boeta$ is the multivariate Euler umbra,
then $2 \bio \equiv \frac{1}{2} (\boeta - \boldsymbol{u}) + \bio.$
\end{lem}   
\begin{proof} 
We have $f(2 \bio, \zbs)= f(\bio, 2 \zbs)= 2(z_1 + \cdots + z_d)/[e^{2(z_1 + \cdots + z_d)} - 1] = f(\boeta - \boldsymbol{u}, \frac{\zbs}{2}) f(\bio, \zbs).$
\end{proof}
\begin{thm} \label{d} {\bf (Relation between multivariate Bernoulli and Euler polynomials)}  We have $ 2^{|\vbs|} \mathcal{B}^{(t)}_{\vbs} \left( \frac{\xbs}{2} \right) = E[\mathcal{E}^{(t)}_{\vbs} (\xbs + t \punt \bio)]$ where $\bio$ is the multivariate Bernoulli umbra.
\end{thm}
\begin{proof} 
From Lemma \ref{c}, we have $t \punt (2 \bio) \equiv t \punt \left[\frac{1}{2} (\boeta - \boldsymbol{u}) + \bio\right] \equiv t \punt \frac{1}{2} (\boeta - \boldsymbol{u}) + t \punt \bio,$ where last equivalence follows form the third equivalence in (\ref{distributive}). The result follows since $2^{|\vbs|} \mathcal{B}^{(t)}_{\vbs} \left( \frac{\xbs}{2} \right) = E[(\xbs  + t \punt (2 \bio))^{\vbs}].$  
\end{proof}
{\bf Conclusions and open questions: multivariate time-space harmonic polynomials.} 
A family of polynomials $\{P(x,t)\}_{t \geq 0}$ is said to be {\it time-space harmonic} with
respect to a stochastic process $\{X_t\}_{t \geq 0}$ if $E[P(X_t,t) \,\, |\; \mathfrak{F}_{s}]
=P(X_s,s),$ for all $s \leq t,$ where $\mathfrak{F}_{s}=\sigma\left( X_\tau : \tau \leq s\right)$
is the natural filtration associated with $\{X_t\}_{t \geq 0}.$ Recently \cite{Dinardo2} the authors have introduced
a new family of polynomials which are time-space harmonic with respect to L\'evy
processes and by which to express all other families of polynomials sharing the same properties. These
polynomials are Appell polynomials and have the form $E[(x + t \punt \alpha)^i]$ for all positive integers $i.$ By  
generalizing the definition of conditional evaluation given in \cite{Dinardo2} to the multivariate 
case, the multivariate Bernoulli and Euler polynomials should result to be time-space harmonic 
with respect to the multivariate L\'evy processes $-t \punt \bio$ and $\frac{1}{2} [t \punt (\boldsymbol{u} - 1 \punt \boeta)]$ respectively, whose probabilistic counterparts could be recovered via Propositions \ref{int1} and \ref{int2}. Indeed, Corollary \ref{tsh} shows that these polynomials share one of the main properties of  time-space harmonic polynomials: when the vector of indeterminates is replaced by the corresponding L\'evy process, their overall mean is zero. We believe that the setting here introduced, together with the one given in \cite{Dinardo2}, could be a fruitful way to build a theory of time-space harmonic polynomials with respect to multivariate L\'evy processes.
\section{Acknowledgements}
We are grateful to the referees for a number of helpful suggestions for improvement in the article.

\end{document}